\documentclass[a4paper, 12pt, final]{amsart}
\usepackage{amscd}
\usepackage{amssymb}
\usepackage[all]{xy}
\usepackage[nohug]{diagrams}
\newcommand{\Ker}{\operatorname{Ker}}

\newcommand{\iso}{\overset{\sim}{\longrightarrow}}
\newcommand{\C}{\mathbb{C}}
 \renewcommand{\P}{\mathbb{P}}
\newcommand{\Q}{\mathbb{Q}}
\newcommand{\Z}{\mathbb{Z}}
\newcommand{\Le}{\mathbb{L}}
\DeclareMathAlphabet{\mathbbb}{U}{bbold}{m}{n}
\newcommand{\un}{\mathbbb{1}}

\newcommand{\sA}{\mathcal{A}}

\newcommand{\sM}{\mathcal{M}}
\newcommand{\sC}{\mathcal{C}}

\newcommand{\rat}{\text{rat}}

\newcommand{\To}{\longrightarrow}
\newtheorem{thm}[equation]{Theorem}
\newtheorem{lm}[equation]{Lemma}

\newtheorem{conj}[equation]{Conjecture}
\newtheorem{prop}[equation]{Proposition}
\theoremstyle{definition}
\newtheorem{defn}[equation]{Definition}
\newtheorem{ex}[equation]{Example}
\newtheorem{exs}[equation]{Examples}
\newtheorem{rk}[equation]{Remark}
\newtheorem{rks}[equation]{Remarks}

\begin{document}

\title[Derived equivalence and motives]{Derived categories of coherent sheaves and  motives  of K3 surfaces} 
 
\author{A. Del Padrone} 
\address{Dipartimento di Matematica\\
Universit\`a degli Studi di Genova\\
Via Dodecaneso 35\\
16146 Genova\\
Italy} 
\email[A. Del Padrone]{delpadro@dima.unige.it}

\author{C. Pedrini}\thanks{The authors are members of INdAM-GNSAGA}
\address{Dipartimento di Matematica\\
Universit\`a degli Studi di Genova\\
Via Dodecaneso 35\\
16146 Genova\\
Italy} 
\email[C. Pedrini]{pedrini@dima.unige.it}

\maketitle

\begin{abstract}  
Let  $X$ and $Y$ be smooth complex projective varieties.  
We will denote  by  $D^b(X)$ and $D^b(Y)$ their derived categories of bounded complexes of coherent sheaves;
$X$ and $Y$ are  derived equivalent if there is  a $\C$-linear equivalence $F \colon D^b(X) \iso D^b(Y)$. 
Orlov conjectured that if $X$ and $Y$ are  derived equivalent then their motives $M(X)$ and $M(Y)$ are isomorphic
in Voevodsky's  triangulated category of motives $DM_{gm}(\C)$ with $\Q$-coefficients.
In this paper we prove the conjecture in the case $X$ is a K3 surface admitting an elliptic 
fibration (a case that always occurs if the Picard rank $ \rho(X)$ is at least 5) 
with  finite-dimensional Chow motive. 
We also relate this result with a conjecture by Huybrechts showing that,  
for a K3 surface with a symplectic involution $f$,  
the finite-dimensionality of its motive implies that $f$ acts as the identity on the Chow group of $0$-cycles. 
We give examples of pairs of K3 surfaces with the same finite-dimensional motive but not derived equivalent.
\end{abstract}

\section {introduction} 

Let $X$ be a smooth projective variety over  $\C$.  
We will denote  by  $D^b(X)$ the derived category of bounded complexes of coherent sheaves on $X$.
We say that two smooth projective varieties $X$ and $Y$ are {\bf derived equivalent} if there is 
a $\C$-linear equivalence $F \colon D^b(X) \iso D^b(Y)$ ([Ro], [B-B-HR]).
It is a fundamental result of Orlov  [Or1, Th. 2.19] that
every such equivalence is a {\bf Fourier-Mukai transform}, 
i.e. there is an object $\sA \in D^b(X \times Y)$, unique up to isomorphism, called its {\bf kernel}, 
such that $F$ is isomorphic to the functor $\Phi_{\sA}:= p_*(q^*(-) \otimes \sA)$, 
where $p_*$,  $q^*$ and $\otimes$ are derived functors. 
Therefore such pairs $X$ and $Y$ are also called {\bf Fourier-Mukai partners}.
Orlov also proved the following Theorem and stated the conjecture below.

\begin{thm}\label{TeoOrDer} ([Or2, Th. 1])
If $\dim X =\dim Y =n$ and $\Phi_{\sA} \colon D^b(X) \To D^b(Y)$ is an exact fully faithful functor 
satisfying the following condition
$$
(*)\quad\textrm{ the dimension of  the support of }\sA \in D^b(X \times Y) \textrm{ is } n,  
$$
then the motive $M(X)_{\Q}$ is a direct summand of $M(Y)_{\Q}$. 
If in addition the functor $\Phi_{\sA}$ is an equivalence  then the motives $M(X)_{\Q}$ and $M(Y)_{\Q}$ are isomorphic 
in Voevodsky's triangulated category of motives  $DM_{gm}(\C)_{\Q}$. 
Moreover the same results hold true at the level of {\it integral} motives.
\end{thm}

\begin{conj}\label{ConjOr} ([Or2, Conj. 1]) 
Let  $X$ and $Y$ be smooth projective varieties and let $F\colon D^b(X) \To D^b(Y)$
be a fully faithfull functor. Then the motive $M(X)_{\Q}$ is a direct summand of $M(Y)_{\Q}$.
If $F$ is an equivalence then the motives $M(X)_{\Q}$ and $M(Y)_{\Q}$ are isomorphic. 
\end{conj}

In [Hu1, 2.7] Huybrechts proved that if $F\colon D^b(X) \simeq D^b(X)$ is a self equivalence then 
it acts identically on cohomology if and only if it acts identically on Chow groups (see section 5). 
This naturally suggests  the following conjecture, which appears in [Hu2, Conj. 3.4].
\begin {conj}\label{ConjHuy}
Let $X$ be a complex K3 surface and let  $f \in \textrm{Aut}(X)$ be a symplectic automorphism, 
i.e. $f^*$ acts as the identity on $H^{2,0}(X)$. 
Then $f^*=id$ on $CH^2(X)$.
\end{conj}
\medskip

\noindent 
In section 2 we recall some results on the finite dimensionality of motives and their Chow-K\"unneth decompositions.
\par
\noindent In section 3, after some general remarks on the derived equivalences between two 
smooth projective varieties $X$ and $Y$, we relate the derived equivalence with ungraded 
motives and finite-dimensionality (Proposition \ref{PropFDStabilityDer}).
\par
\noindent In section 4 we specialize to the case of K3 surfaces $X$ and $Y$ and prove our main result (Theorem \ref{TeoDerK3}): 
Orlov's conjecture holds true for K3 surfaces $X$ and $Y$ if the motive of $X$ is finite-dimensional  and $X$ admits an elliptic fibration, 
a case that always occurs if  the Picard rank $\rho(X)$ is at least 5. 
This restriction can possibly be removed, according to a claimed result by Mukai in [Mu2, Th2].
\par
\noindent In section 5 we consider the case of a K3 surface with a symplectic involution $\iota$ and prove (Theorem \ref{TeoNik})  
that Huybrechts' Conjecture \ref{ConjHuy} holds true for $f=\iota$ if $X$ has a finite-dimensional motive. 
We also show (Theorem \ref{TeoRho9} and Examples \ref{ExamplesVG-S}) the existence of K3 surfaces $X$ and $Y$ 
which are not derived equivalent but
with isomorphic motives.
\medskip

\noindent 
{\bf Akwnoledgements.}
We thank Claudio Bartocci for many helpful comments on a early draft of this paper.

\section {Categories of motives  and finite dimensionality}
Let $X$ be a smooth variety over a perfect field $k$ and let $CH^{i}(X)$ be 
the Chow group of cycles of codimension $i$ modulo rational equivalence. 
We will denote by $A^{i}(X)= CH^{i}(X)_{\Q}$ the $\Q$-vector space 
$CH^{i}(X)\otimes_{\Z}\Q$.
 
\subsection{Pure motives}\label{SubsecPM}
Let $\sM^{eff}_{rat} (k)$ be the {\it covariant} pseudo-abelian, tensor,  additive category of 
{\bf effective Chow motives}  with $\Q$-coefficients over a perfect field $k$. 
Its objects are couples $(X,p)$ where  $X$ is  a smooth projective variety 
and $p \in CH_{\dim\, X}(X \times X)_{\Q}$ is a projector, i.e. $p\circ p =p^2 =p$. 
Morphisms  between $(X,p)$ and $(Y,q)$ in $\sM^{eff}_{rat}$ are given by  correspondences 
$\Gamma \in A_{\dim\, X}(X \times Y)$.
More precisely:
$$
\textrm{Hom}_{\sM^{eff}_{rat}(k)}((X,p),(Y,q))=q\circ CH_{\dim\, X}(X\times Y)_{\Q}\circ p.
$$
The motive of a smooth projective variety $X$ is defined as  $\mathfrak{h}(X) = (X,\Delta_X)\in  \sM^{eff}_{rat}(k)$, thus giving
a covariant monoidal functor $h\colon \mathcal{S}mProj/k\To \sM^{eff}_{rat}(k)$ 
which sends $f\colon X\To Y$ to its graph $\mathfrak{h}(f)=[\Gamma_{f}]\colon \mathfrak{h}(X)\To \mathfrak{h}(Y)$.
Let $X =\P^1$, then the structure map $X \to \mathrm{Spec}(k) $ together with the inclusion 
of a closed point $P \in  \P^1$ (eventually defined over an algebraic extension of $k$, see [K-M-P, 7.2.8]) 
induces  a splitting
$$
\mathfrak{h}(\P^1) \simeq \un \oplus \mathbb L
$$
where $\un=(\mathrm{Spec}(k), \Delta_{\mathrm{Spec}(k)})\simeq (\P^{1}, [\P^{1}\times P])$ 
is the unit of the tensor structure and $\mathbb{L}=(\P^{1}, [P\times \P^{1}])$ is  the {\bf Lefschetz motive}. 
By $\sM_{rat}(k)$ we will denote the tensor category of {\bf covariant Chow motives}, 
obtained from $\sM^{eff}_{rat} (k)$ by inverting $\mathbf L$, as in [K-M-P].
\smallskip

\noindent
We will also consider  the $\Q$-linear rigid tensor category of {\bf ungraded covariant Chow motives} 
$\mathcal{UM}_{rat}(k)$ (see for example [Ma, \S 2, \S 3, p. 459] and [D-M, 1.3]).
It is the pseudo-abelian hull of the $\Q$-linear additive category of ungraded correspondences. 
Hence, its objects are pairs $(X,e)$ with $X$ a smooth projective variety,  
$e \in CH_{*}(X \times X)_{\Q} =\oplus_{i=0}^{2\dim\, X}CH_{i}(X \times X)_{\Q}$ 
a projector, and 
$$
\textrm{Hom}_{\mathcal{UM}_{rat}(k)}((X,e),(Y,f))=f\circ CH_{*}(X\times Y)_{\Q}\circ e;
$$
the ungraded motive of $X$ is $\mathfrak{h}(X)_{\textrm{un}}:=(X,\Delta_{X})$; its endomorphism algebra is
the $\Z$-graded ring (w.r.t. composition of correspondences, see [Ma, \S4 p. 452])
$$
\textrm{End}_{\mathcal{UM}_{rat}(k)}(\mathfrak{h}(X)_{\textrm{un}})=CH_{*}(X\times X)_{\Q}.
$$
$\mathcal{UM}_{rat}(k)$ is a rigid $\Q$-linear tensor category in the obvious way.
\medskip

\subsection{Mixed motives}
Let $DM^{eff}_{gm}(k)$ be the triangulated  category of {\bf effective  geometrical motives} 
constructed by Voevodsky in [Voev].
We recall that there is a covariant functor $M \colon  Sm/k \to DM^{eff}_{gm}(k)$ where $Sm/k$
 is the category of smooth schemes of finite type over $k$. 
We shall write $DM_{gm}^{eff}(k,\Q)$ for the pseudo-abelian hull of the category obtained 
from $DM^{eff}_{gm}(k)$ by tensoring morphisms with $\Q$, and usually abbreviate 
it into $DM_{gm}^{eff}(k)$. 
Then $M$ induces a covariant  functor 
$$
\Phi \colon  \sM^{eff}_{rat}(k) \to  DM_{gm}^{eff}(k)
$$
which is a full embedding. 
We will denote by $DM_{gm}(k)=DM_{gm}(k,\Q) $ the category obtained from 
$DM_{gm}^{eff}(k)$ by inverting  the image $\Q(1)$ of $\mathbf L$. 
Hence, for two smooth projective varieties $X$ and $Y$, $\mathfrak{h}(X) \simeq \mathfrak{h}(Y)$ in $\sM_{rat} (k)$  
if and only if the images $M(X)$ and $M(Y)$ are isomorphic in $DM_{gm}(k)$.
\smallskip

\subsection{Finite-dimensional motives}
We now recall several notion of ``finiteness'' on motives (see [Ki, 3.7], [Maz, 1.3], [An1, 12] and [An2, 3]).
Let $\sC$ be a pseudoabelian, $\Q$-linear,  symmetric tensor category and let $A$ be an object in $\sC$. 
Thanks to the symmetry isomorphism of $\sC$ the symmetric group on $n$ letters $\Sigma _n$
acts naturally on the $n$-fold tensor product $A^{\otimes n}$ of $A$ by itself for each object $A$:
any $\sigma \in \Sigma_n$ defines a map 
$\sigma_{A^{\otimes n}} : A^{\otimes n} \to A^{\otimes n}$. 
We recall that there is a one-to-one correspondence between all irreducible representations
of the group $\Sigma_n$ (over $\Q $) and all partitions of the integer $n$.
Let $V_{\lambda }$ be the irreducible representation corresponding
to a partition $\lambda $ of $n$ and let $\chi _{\lambda }$ be the
character  of the representation $V_{\lambda }$, then
$$
d_{\lambda }= 
\frac{\dim (V_{\lambda })}{n!}\sum_{\sigma \in
\Sigma _n}\chi _{\lambda }(\sigma )\cdot\sigma\in\Q\Sigma_{n} 
$$
gives, when $\lambda$ varies among the partitions of $n$, a set of pairwise orthogonal central 
(non primitive) idempotents in the group algebra $\Q\Sigma_{n}$; the two-sided ideal
$(d_{\lambda})=d_{\lambda}\Q\Sigma_{n}$ is the isotypic component of $V_{\lambda}$ inside
$\Q\Sigma_{n}$ hence $(d_{\lambda})\cong V_{\lambda}^{\lambda}$ as $\Q\Sigma_{n}$-modules.  
Let
$$ 
d^{A}_{\lambda }= 
\frac{\dim (V_{\lambda })}{n!}\sum_{\sigma \in
\Sigma _n}\chi _{\lambda }(\sigma )\cdot\sigma_{A^{\otimes n}}
\in\mathrm{Hom}_{\sC}(A^{\otimes n},A^{\otimes n})
$$
where $\sigma_{A^{\otimes n}}$ is the morphism associated to $\sigma$.
Then $\{ d^{A}_{\lambda }\} $ is a set of pairwise orthogonal idempotents in 
$\mathrm{Hom}_{\sC}(A^{\otimes n},A^{\otimes n})$ such that 
$\sum d^{A}_{\lambda }=\textrm{Id}_{A^{\otimes n}}$. 
The category $\sC$ being pseudoabelian, they give a functorial decomposition 
$$
A^{\otimes n}=\oplus_{|\lambda|=n}S_{\lambda}(A)\quad (S_{\lambda}(A)=\textrm{Im}\,d^{A}_{\lambda}),
$$
where $S_{\lambda}$ is the {\bf isotypic Schur functor} associated to $\lambda$ 
(which is a just ``multiple'' of the classical one).
The $n$-th   symmetric product $\textrm{Sym}^nA$ of $A$ is  then defined to be  the image 
$\mathrm{Im}(d^{A}_{\lambda})$ when $\lambda$ corresponds to the partition $(n)$, and
the  $n$-th exterior power $\wedge^n  A$ is $\mathrm{Im}(d^{A}_{\lambda})$
when $\lambda $ corresponds to the partition $(1,\dots, 1)$.
If $\sC =\sM_\rat(k)$ and $A=\mathfrak{h}(X) \in \sM_\rat(k)$  for  a smooth  projective variety $X$, 
then  $\wedge^n A $ is the image of $\mathfrak{h}(X^n)=\mathfrak{h}(X)^{\otimes n}$ under the projector
$(1/ n!) (\sum_{\sigma \in \Sigma_n} \mathrm{sgn}(\sigma)\Gamma_{\sigma})$, 
while $\textrm{Sym}^n A$ is its image under the projector
$(1/ n!) (\sum_{\sigma \in \Sigma_n}\Gamma_{\sigma})$.
\defn \label{DefFD}
The object $A$ in $\sC$ is said to be {\bf Schur finite} if $S_{\lambda}(A)=0$ for some 
partition $\lambda$ (i.e. $d^{A}_{\lambda}=0$ in $\textrm{End}_{\sC}(A^{\otimes})$); 
it is said to be {\bf evenly (oddly) finite-dimensional} 
if $\wedge^n A=0$ ($\textrm{Sym}^n A=0$) for some $n$.
An object $A$ is {\bf finite-dimensional} (in the sense of Kimura and O'Sullivan) 
if it can be decomposed into a direct sum $A_+\oplus A_-$ where 
$A_+$ is evenly finite-dimensional and 
$A_-$ is oddly finite-dimensional.
\medskip

\noindent
If $A$ is evenly and oddly finite-dimensional then $A=0$ (see [Ki, 6.2] and [An2, 6.2]). 

\begin{rk} 
From the definition it follows that,  for  a smooth projective variety $X$ over $k$,  
the motive $\mathfrak{h}(X)$ is  finite-dimensional in $\sM_{rat}(k)$ if and only if 
$ M(X)$ is finite-dimensional in $DM_{gm}(k)$. 
\end{rk}

Kimura's nilpotence Theorem in [Ki, 7.5] says that if $M$ is finite-di\-men\-sio\-nal, 
any numerically trivial endomorphism {\bf universally of trace zero}
(i.e. given by a correspondence which is numerically trivial as an algebraic cycle) 
of $M$ is nilpotent; therefore 

\begin{thm}\label{TeoConservativity}(Kimura)
If $M$ and $N$ are two finite-dimensional Chow motives and $f\colon M\To N$
is a morphism, then $f$ is an isomorphism if and only if its reduction 
modulo numerical equivalence is such (see [An2 3.16.2)]).
\end{thm}

\noindent
In particular, if $M \in \sM_\rat$ is  a finite-dimensional  motive such that $H^*(M)=0$, 
where $H^{*}$ is any Weil cohomology, 
then $M=0$ ([Ki, 7.3]).

\begin{rk}\label{RemSchurNotKimura}
For Schur-finite objects such a nilpotency result holds 
only under some extra assumptions as shown in [DP-M1] and [DP-M2],
but not in general. In fact  let $\mathcal{C}$ be the $\Q$-linear rigid tensor category of  
bounded chain complexes of finitely generated $\Q$-vector spaces 
with the usual tensor structure and the ``Koszul'' commutativity constraint. 
Then $\textrm{Id}_{\Q}\colon \Q\To\Q$ can be thought of as an object $A$  of $\sC$, 
 concentrated in homological degrees $1$ and $0$.
It is indecomposable as $\textrm{End}_{\sC}(A)\cong \Q$,
and it is not finite-dimensional for $\wedge^{n}(A)\neq 0$ and 
$\textrm{Sym}^{n} (A)\neq 0$ (as complexes) for each $n\in \mathbb{N}$. 
On the other hand $S_{(2,2)}(A)=0$, i.e. $A$ is Schur-finite, 
for it is so under the obvious faithful (but not full) 
$\Q$-linear tensor functor towards $\Z/2$-graded $\Q$-vector spaces. 
Moreover, due to the Koszul rule, $\textrm{Id}_{A}$ is universally of trace zero 
but not nilpotent.
\end{rk}

\begin{exs}
 \item{\bf (1)} 
Finite-dimensionality and Schur-finiteness are  stable under direct sums, tensor products, and direct summand. 
More precisely: $S_{\lambda}(B)=0$ whenever $B$ is a direct summand of $A$ with $S_{\lambda}(A)=0$. 
It is also true that a direct summand of a finite-dimensional object is such ([An2, 3.7]).
Finite-dimensionality implies Schur-finiteness, but the converse does not hold not even in $DM_{gm}(k)$. 
In fact Peter O'Sullivan showed that there exist smooth surfaces $S$ whose motives in $DM_{gm}(k)$ is Schur-finite 
but not finite dimensional, see [Maz, 5.11].
\smallskip

\item{\bf (2)}
Clearly we have $\wedge^{2}\un=0$ in any {\it symmetric} tensor category.
It is also straightforward that $\wedge^{2}\Le =0$ for the Lefschetz motive, and
$\wedge^{3}\mathfrak{h}(\P^{1}) =0$. Kimura showed $\textrm{Sym}^{2g+1}(\mathfrak{h}^{1}(C))=0$
for any smooth projective curve $C$ of genus $g$ [Ki, 4.2].
\end{exs}

We also have Kimura's conjecture:

\begin{conj} 
Any motive in $\sM_\rat$ is finite-dimensional.
\end{conj}

\begin{rk}
The status of the conjecture is the following.
\smallskip

\item{\bf (1)}
The conjecture is true for curves, abelian varieties, Kummer surfaces,  complex  surfaces not of general type with $p_g =0$ 
(e.g. Enriques surfaces), Fano 3-folds [G-G].  
For a  complex surface $X$ of general type with $p_g(X)=0$ the finite-dimensionality of the motive $\mathfrak{h}(X)$ 
is equivalent to Bloch's conjecture, i.e. to the vanishing of the Albanese Kernel of $X$ (see [G-P, Th. 7]).
If the conjecture holds for $\mathfrak{h}(X)$ then it holds true for $\mathfrak{h}(Y)$ with $Y$ a smooth projective variety dominated by $X$. 
The full subcategory of $\sM_\rat$ on finite-dimensional objects is a $\Q$-linear rigid tensor subcategory closed under direct summand.
\smallskip

\item{\bf (2)} Let $X$ be a K3 surface; then   $\mathfrak{h}(X)$ is finite-dimensional in the following cases, see [Pe3]
\begin{itemize}
\item
$\rho(X)=19$ or $\rho(X)=20$. In these cases $X$ has a {\it Nikulin involution} which gives a {\it Shioda-Inose structure},
 in the sense of [Mo, 6.1], and the transcendental motive $\mathfrak{t}_{2}(X)$ of $X$ (see \ref{rCK}) is isomorphic to the transcendental motive of a {\it Kummer
surface} [Pe3, Th. 4]. 
\smallskip

\item 
$X$ has a non-symplectic group $G$ acting trivially on the algebraic cycles and the order of the 
kernel (a finite group) of the map $\textrm{Aut}(X)\To \mathcal{O}(\textrm{NS}(X))$ is different from $3$, where 
$ \mathcal{O}(\textrm{NS}(X))$ is the group of isometries of $NS(X)$. 
Then, by a result in [L-S-Y, Th. 5],  $X$ is dominated by a {\it Fermat surface} $F_{n}$, whose motive 
is of abelian type (hence finite-dimensional) by the {\it Shioda-Katsura inductive structure} [S-K, Th. I].
K3 surfaces satisfying these conditions have  $\rho(X)=2,4,6,10,12,16,18,20$.
\end{itemize}
\smallskip

By a result of Deligne ([De, 6.4]),  for every complex polarized K3 surface there exists a smooth family of polarized K3 surfaces 
$\{X\}_{t \in \Delta}$, with $\Delta$ the unit disk, such that the central fibre $X_0$ is isomorphic to $X$. 
Therefore the finite-dimensionality of the motive of a general K3 surface, i.e. with $\rho(X)=1$, implies the finite-dimensionality 
of the motive of any K3 surface, see [Pe1, 4.3].
\smallskip

\item{\bf (3)}
In all the known cases where the motive $\mathfrak{h}(X)$ is finite-dimensional,  it lies in the tensor subcategory of 
$\sM_{rat}(k)$ generated by the motives of abelian varieties (see [An, 2.5]).
\end{rk}

The following result will appear in [DP].
\begin{prop}\label{PropFDStability}
Let $M=(X,p)$ be an effective Chow motive. Then
\item{\bf (a)} 
The (graded) motive $M$ is Schur-finite if and only if the ungraded motive $M_{\textrm{un}}$
is such. More precisely  for any partition $\lambda$ we have $S^{\sM_{rat}(k)}_{\lambda}(M)=0$ if and only if 
$S^{\mathcal{UM}_{rat}(k)}_{\lambda}(M_{\textrm{un}})=0$.
In particular, being $M$ even or odd depends only on the ungraded isomorphism
class of the ungraded motive $M_{\textrm{un}}$.
\smallskip

\item{\bf (b)} 
If $M$ is finite-dimensional in $\sM_{rat}(k)$ then $M_{\textrm{un}}$ is so in $\mathcal{UM}_{rat}(k)$.
Moreover, if $M=\mathfrak{h}(X)$ with $X$ a variety such that the projections on the even and the odd part 
of the cohomology (w.r.t. a given Weil cohomology theory) 
are algebraic then $\mathfrak{h}(X)$ is finite-dimensional if and only if $\mathfrak{h}(X)_{\textrm{un}}$ is.
\end{prop}

\begin{rk}
The hypothesis in {\bf (b)} of Proposition \ref{PropFDStability} is Jannsen's 
{\it homological sign conjecture} $C^{+}(X)$ [An2, 5.1.3],
called $S(X)$ in [Ja, 13.3].
\end{rk}

\subsection{The refined Chow-K\"unneth decomposition}\label{rCK}
Let for simplicity $k=\C$ in what follows.
We recall from [K-M-P,  2.1] that the covariant Chow motive  $\mathfrak{h}(S) \in \sM_{rat} (\C)_{\Q}$
of any  smooth projective surface $S$  has a  {\bf refined Chow-K\"unneth decomposition}
$$
\sum_{0 \le i \le 4}\mathfrak{h}_i(S)
$$
corresponding to a splitting $\Delta_S =\sum_{0 \le i \le 4 }\pi_i$ of the diagonal in $H^*(S \times S)$. 
Here $\mathfrak{h}_0(S)=(S,[S \times P])\simeq (\textrm{Spec}(\C),\textrm{Id}) = \un$ and 
$ \mathfrak{h}_4(S) =(S,[P \times S]) \simeq \mathbb{L}^2$,
where $P$ is a rational  point  on $S$.
Also 
$$
\mathfrak{h}_2(S) = {\mathfrak{h}_{2}}^\textrm{alg}(S) \oplus \mathfrak{t}_2(S)
$$
with 
${\mathfrak{h}_2}^\textrm{alg}(S) = (S,\pi^{\textrm{alg}}_2)$ 
the effective Chow motive defined by the idempotent
$$
\pi^{\textrm{alg}}_2(S) =\sum_{1 \le h \le \rho} \frac {[D_h \times D_h]}{D^2_h}\in A_{2}(S\times S)
$$
where $\rho=\rho(S)$ is the rank of the Neron-Severi $\textrm{NS}(S)$ 
and $\{D_h\}$ is an orthogonal bases of $\textrm{NS}(S)_{\Q}$.
It follows that ${\mathfrak{h}_2}^{\textrm{alg}}(S) \simeq \mathbb{L}^{\oplus \rho}$. 

\defn \label{DefTP}
The Chow motive  $\mathfrak{t}_2(S)= (S,\pi^{\textrm{tr}}_2,0)$, 
with $\pi^{\textrm{tr}}_2 =\pi_2 -\pi^{\textrm{alg}}_2$, 
is  the {\bf transcendental  part} of the motive $\mathfrak{h}(S)$. 
Then $H^i(\mathfrak{t}_2(S))=0$ if $i \ne 2$ and 
$ H^2(\mathfrak{t}_2(S)) =H^2_{\textrm{tr}}(S) =\pi^{\textrm{tr}}_2 H^2(S,\Q) = H^2_{\textrm{tr}}(S,\Q)$. 
\medskip

The Chow motive $\mathfrak{t}_{2}(S)$ does not depend on the choices made to define the refined Chow-K\"unneth decomposition,
it is functorial on $S$ for the action of correspondences, and it is a {\it birational invariant} of $S$ (see [K-M-P]).

\begin{rk}
For any smooth projective surface $S$, all the motives $\mathfrak{h}_i(S)$ appearing in a refined  Chow-K\"unneth decomposition, 
except possibly  for $\mathfrak{t}_2(S)$ are finite dimensional.
Therefore the motive $\mathfrak{h}(S)$  of a surface $S$ is finite dimensional if and only if 
the motive $\mathfrak{t}_2(S)$ is evenly finite dimensional, i.e. $\wedge^n \mathfrak{t}_2(S) =0$ for some $n$.
If $S$ has no irregularity (i.e. $q(S):=\dim\, H^{1}(S,\mathcal{O}_{S})=0$) then $\mathfrak{h}_{1}(S)=\mathfrak{h}_{3}(S)=0$.
\end{rk}

\subsection{Refined C-K decomposition of a K3 surface}\label{SubSecRCK}
Let now $S$ be a smooth (irreducible) projective  K3 surface  over $\C$.
As $S$ is a regular surface (i.e. $q(S)=0$),  
its  refined Chow-K\"unneth decomposition  has the following shape
$$
\mathfrak{h}(S)
\;\;=\;\;
\un
\oplus\; 
{\mathfrak{h}_{2}}^\textrm{alg}(S) \oplus \mathfrak{t}_2(S)
\;\oplus 
\mathbb{L}^{\otimes 2}
\;\;\simeq\; \;
\un\oplus\; 
\mathbb{L}^{\oplus \rho }\oplus \mathfrak{t}_2(S)
\;\oplus 
\mathbb{L}^{\otimes 2}
$$
with  $1 \le  \rho  \le 20$.
Moreover   
$$
A_i(\mathfrak{t}_2(S))=\pi^{\textrm{tr}}_2 A_i(S)=0    \textrm{  for  }    i \ne 0;
\qquad   A_0(\mathfrak{t}_2(S)) =A_0(S)_0,
$$    
where the last $\Q$-vector space is the group of $0$-cycles of degree $0$ tensored with $\Q$. 
We also have 
$$ 
\dim H^2(S) = b_2(S) = 22; \qquad  \dim H^2_{\textrm{tr}}(S) = b_2(S)  - \rho(S)=22-\rho.  
$$
By $T_{S,\Q}=H^{2}_{\textrm{tr}}(S,\Q)$ we  will denote the {\it lattice of transcendental cycles}, tensored with $\Q$,  
it coincides with the orthogonal complement  to the Neron-Severi $\textrm{NS}(S)\otimes \Q$ in $H^2(S,\Q)$.

\section{Derived equivalence and motives}
Let $X$ and $Y$ be smooth projective varieties over $\C$.
If $X$ and $Y$ are derived equivalent then 
(see e.g. [Ro], [Hu], [B-B-HR])
$\dim\, X=\dim\, Y$, $\kappa(X)=\kappa(Y)$ (where $\kappa$ is the Kodaira dimension), and
$H^{*}(X,\Q)\simeq H^{*}(Y,\Q)$ (isomorphism of $\Z/2$-graded vector spaces).
If $\dim\, X=2$ the surfaces $X$ and $Y$ have the same Picard number and 
the same topological Euler number; 
and $X$ is a K3 surface, respectively an abelian surface, if and only if $Y$ is.

Kawamata conjectured that, up to isomorphism, 
$X$ has only a finite number of Fourier-Mukai partners $Z$ [Ka].
This conjecture is true for curves (and in this case $Z\simeq X$, [B-B-HR, 7.16]),
surfaces ([B-M]), abelian varieties (see [Ro, 3]  and [H-NW, 0.4]), and
varieties with ample or antiample canonical bundle, in which case $Z\simeq X$ 
(due to Bondal-Orlov, see [B-B-HR, 2.51]).
\medskip

The following result is somewhat in the same spirit, with respect to the relation between 
derived equivalence of smooth projective varieties and their associated Chow motives.

\begin{prop}\label{PropFDStabilityDer}
Let $\Phi_{\mathcal{A}}\colon D^{b}(X)\To D^{b}(Y)$ an exact equivalence, then 

\item{\bf (a)} The ungraded Chow motives $\mathfrak{h}(X)_{\textrm{un}}$ and $\mathfrak{h}(Y)_{\textrm{un}}$ are isomorphic.
If the condition $(*)$ in Theorem \ref{TeoOrDer} is satisfied then the isomorphism is given by a correspondence of degree zero, hence 
$\mathfrak{h}(X)$ and $\mathfrak{h}(Y)$ are isomorphic as Chow motives.
\smallskip

\item{\bf (b)}  The (graded) motive $\mathfrak{h}(X)$ is Schur-finite if and only if $\mathfrak{h}(Y)$ is such.
\smallskip

\item{\bf (c)}  If $X$ is curve, a surface, an abelian variety, or a finite product of them 
(or any variety if $k$ is algebraic over a finite field), then 
$\mathfrak{h}(X)$ is finite-dimensional if and only if $\mathfrak{h}(Y)$ is such.
\end{prop}

\begin{proof}
{\bf (a)} 
The argument in [Or 1, p. 1243], which we briefly recall can be used to prove that 
$\mathfrak{h}(X)_{\textrm{un}}\cong \mathfrak{h}(Y)_{\textrm{un}}$  in $\mathcal{UM}_{rat}(k)$.
Let  $\mathcal{B}\in D^{b}(X\times Y)$ be the kernel of the  quasi-inverse of 
$\Phi_{\mathcal{A}}$. 
Using Huybrechts' notation ([Hu1, p. 1534] and [Hu2, 4.1]),
we then have (non homogeneus, $\Q$-linear) algebraic cycles
$$
a=v^{\textrm{CH}}(\mathcal{A}):=
\textrm{ch}(\mathcal{A})\cdot\sqrt{\textrm{td}_{X\times Y}}=
p^{*}_{1}\left(\sqrt{\textrm{td}_{X}}\,\right)\cdot 
\textrm{ch}(\mathcal{A})\cdot 
p^{*}_{2}\left(\sqrt{\textrm{td}_{Y}}\,\right)
\in CH_{*}(X\times Y)_{\Q},
$$
and
$$
b=v^{\textrm{CH}}(\mathcal{B})=
p^{*}_{1}\left(\sqrt{\textrm{td}_{Y}}\,\right)\cdot 
\textrm{ch}(\mathcal{B})\cdot 
p^{*}_{2}\left(\sqrt{\textrm{td}_{X}}\,\right)
\in CH_{*}(Y\times X)_{\Q},
$$
where $\textrm{td}$ is the Todd class and $\textrm{ch}\colon D^{b}(Z)\To CH_{*}(Z)_{\Q}$
is the composition of the Chern character with the Euler characteristic 
$\chi(\mathcal{E})=\sum (-1)^{i}[\mathcal{H}^{i}(\mathcal{E})]\in K_{0}(Z)$ of the complex of sheaves $\mathcal{E}$. 
Orlov proved, by Grothendieck-Riemann-Roch, that 
$$
b\circ a=[\Delta_{X}]=\textrm{Id}_{\mathfrak{h}(X)_{\textrm{un}}},
\quad\textrm{ and }\quad 
a\circ b=[\Delta_{Y}]=\textrm{Id}_{\mathfrak{h}(Y)_{\textrm{un}}}
$$
as (ungraded) correspondences. 

In case the kernel $\mathcal{A}$ satisfies the hypothesis $(*)$ of Theorem \ref{TeoOrDer}, 
that is $\dim \textrm{supp}(\mathcal{A})=\dim\, X$, it turns out that the ``middle components'' 
$a_{d}\in CH_{d}(X\times Y)_{\Q}$ and $b_{d}\in CH_{d}(Y\times X)_{\Q}$ of the above cycles 
$a$ and $b$ (which are correspondences of degree zero) give an isomorphism at the level of usual Chow motives.
\smallskip

{\bf (b)} 
As already observed in Proposition \ref{PropFDStability}, being Schur-finite for a graded motive $M$ 
can be tested on $M_{\textrm{un}}$.
\smallskip

{\bf (c)} 
In all these cases $C^{+}(X)$ holds true, hence Proposition \ref{PropFDStability} {\bf (b)} applies. 
\end{proof}

\begin{ex}
Let $X =A$ be an abelian variety, $Y=\widehat{A}$  its dual and let 
$\mathcal{A}=\mathcal{P}_{A}\in\textrm{Pic}(A\times \widehat{A})$
be the sheaf complex given by the Poincar\'e bundle.  
The corresponding isomorphism of ungraded Chow motives
is given by
$$
\textrm{ch}(\mathcal{P}_{A})\colon
h(A)_{\textrm{un}}\To h(\widehat{A})_{\textrm{un}}
$$
because the Todd classes are $1$ for abelian varieties. 
It can be shown   (see [B-L 16.3]) that it coincides with the 
{\it motivic Fourier-Mukai transform} of  Deninger and Murre ([D-M, 2.9]).
We note that in this case the dimension of the support of $\mathcal{A}$ 
is equal to $\dim (A\times \widehat{A})=2\cdot\dim\, A$.
As $A$ and $\widehat{A}$ are isogenus it follows that their Chow motives (with $\Q$-coefficients) 
are isomorphic (see for example [An1, 4.3.3]). 
\end{ex}

\begin{rks}\label{RemOrHp(*)}
Let us make two comments on Orlov's hypothesis $(*)$, that is ``the dimension of the support of the kernel 
$\mathcal{A}$ of the equivalence $D^{b}(X)\simeq D^{b}(Y)$ equals $\dim\, X$''.
\smallskip

\item{\bf(1)}
If $\Phi_{\mathcal{A}}$ is an equivalence then the natural projections 
$$\textrm{supp}(\mathcal{A})\To X,\quad
\textrm{supp}(\mathcal{A})\To Y
$$ 
are surjective [Hu, 6.4]. Therefore, in general, 
$\dim \textrm{supp}(\mathcal{A})\geq\dim\, X$
whenever $\Phi_{\mathcal{A}}$ is an equivalence.
\smallskip

\item{\bf (2)}
If $\Phi_{\mathcal{A}}$ is an equivalence and Orlov's hypothesis $(*)$ holds true then 
$X$ and $Y$ are {\bf K-equivalent}, a notion due to Kawamata [Ka] (see [B-B-HR, 2.48]).
In case $X$ and $Y$ are smooth projective complex {\it surfaces},
they are K-equivalent if and only if they are isomorphic [B-B-HR, 7.19].
This is, in general, not the case for K3 surfaces, see for example [So].
\end{rks}

In connection with the result in [Or2, Th. 1] Orlov made the following more precise conjecture [Or2, Conj. 2]:

\begin{conj}\label{Conj2Or}
Let  $\mathcal{A}$ be an object on $X\times Y$ for which $\Phi_{\sA} \colon D^b(X) \To D^b(Y)$ is an equivalence. 
Then there are line bundles $L$ and $M$ on $X$ and $Y$, respectively, such that the $\dim\, X$ component 
of the cycle associated to $\mathcal{A}^{\prime}:=p_{1}^{*}L\otimes \sA \otimes p_{2}^{*}M$ 
determines an isomorphism between the motives $M(X)_{\Q}$ and $M(Y)_{\Q}$ in $DM_{gm}(\C)_{\Q}$. 
\end{conj}

\section{Derived equivalence and complex K3 surfaces}  

Let us now consider Orlov's Conjecture \ref{ConjOr} in low dimension; a case of particular interest is that of K3 surfaces. 
We recall that if $Y$ is a Fourier-Mukai partner of a K3 surface $X$ (respectively abelian surface), 
then also $Y$ is a K3 surface (respectively abelian surface). 
\medskip

We fix some notation. 
For a K3, or abelian, smooth projective complex surface $X$ we have the {\bf Mukai lattice},
also called {\bf extended Hodge lattice} in [B-M, 5], that is the cohomology ring
$$
\widetilde{H}(X,\Z):=H^{0}(X,\Z)\oplus H^{2}(X,\Z)\oplus H^{4}(X,\Z), 
$$
endowed with the symmetric bilinear form 
$$
\left<(r_{1},D_{1},s_{1}),(r_{1},D_{1},s_{1})\right>:=D_{1}\cdot D_{2}-r_{1}s_{2}-r_{2}s_{1},
$$
and the following Hodge decomposition
$$
\widetilde{H}^{(0,2)}(X,\C)=H^{0,2}(X,\C),\quad \widetilde{H}^{(2,0)}(X,\C)=H^{2,0}(X,\C),
$$
$$
\widetilde{H}^{(1,1)}(X,\C)=H^{0}(X,\C)\oplus H^{1,1}(X,\C)\oplus H^{4}(X,\C).
$$
Inside $H^{2}(X,\Z)$ we have two sublattices, the {\bf Neron-Severi lattice} 
$$
\textrm{NS}(X)=H^{2}(X,\Z)\cap H^{1,1}(X,\C),
$$
and its orthogonal complement $T_{X}$, the {\bf transcendental lattice} of $X$.
The transcendental lattice inherits a Hodge structure from $H^{2}(X,\Z)$. 
\smallskip

\begin{defn}
Let $X$ and $Y$ be two complex K3 surfaces. 
A map $T_{X} \to T_{Y}$ (resp. $T_{X,\Q } \to T_{Y,\Q}$) is a {\bf Hodge homorphism  of} 
(resp. {\bf rational}) {\bf Hodge structures} 
if  it preserves the Hodge structures  of  $H^2_{tr}(X) \otimes \C$ and  of $H^2_{tr}(Y) \otimes \C$,  
i.e. if the one dimensional subspace $ H^{2,0}(X) \subset T_X \otimes \C$ goes to $ H^{2,0}(Y) \subset T_Y \otimes \C$. 
A Hodge isomorphism $T_{X } \to T_{Y}$ is an  {\bf Hodge isometry} if it is an isometry with respect 
to the quadratic form induced by the usual intersection form. 
A rational Hodge isometry $\phi\colon T_{X,\Q } \to T_{Y,\Q}$
is {\bf induced by an algebraic cycle} $\Gamma  \in CH_{2}(X \times Y)_{\Q}$ if  
$\phi=\Gamma_* : T_{X,\Q } \to T_{Y,\Q}$
(cf. [Mu, pp. 346-347]).
\end{defn}

Due to work of Mukai and Orlov ([Mu], [Or1, 3.3 and 3.13], [B-M, 5.1])
we have the following result:

\begin{thm}\label{TeoOrK3}
Let $X$ and $Y$ be a pair of K3 (resp. abelian) surfaces. The following statements are equivalent.

\item{\bf (a)}
$X$ and $Y$ are derived equivalent,
\smallskip

\item{\bf (b)}
the transcendental lattices $T_{X}$ and $T_{Y}$ are Hodge isometric,
\smallskip

\item{\bf (c)}
the extended Hodge lattices $\widetilde{H}(X,\Z)$ and $\widetilde{H}(Y,\Z)$ are Hodge isometric,
\smallskip

\item{\bf (d)} 
Y is isomorphic to a fine, two-dimensional moduli space of stable sheaves on $X$.
\end{thm}

The next result relates the finite-dimensionality of the motive of a K3 surface with Orlov's conjecture.

\begin {thm}\label{TeoDerK3} 
Let $X,Y$ be smooth projective K3 surfaces over $\C$ 
such that $X$ has an elliptic fibration
and the Chow motive $\mathfrak{h}(X)$ is finite dimensional.
If  $D^b(X) \simeq D^b(Y)$ then the motives 
$M(X)$ and $M(Y)$ are isomorphic in $DM_{gm}(\C)$.
\end{thm} 
  
\begin {proof} 
By point {\bf (b)} of Proposition \ref{PropFDStabilityDer} we know that $\mathfrak{h}(Y)$ is finite-dimensional.
Theorem \ref{TeoOrK3} ensures the existence of a Hodge isometry $\phi :  T_{X,\Q} \iso T_{Y,\Q}$ 
which, by [Ni, Th. 3],  is induced by an algebraic cycle,  i.e. there exists an algebraic correspondence 
$\Gamma \in CH_{2}(X \times Y)_{\Q}$ such that $\Gamma_*=\phi$. 
Then $\pi^{Y}_{2}\circ\Gamma\circ \pi^{X}_{2}$ induces an isomorphism between
the transcendental motives as homological motives, hence numerical ones;
thus, thanks to Theorem \ref{TeoConservativity}, it is an isomorphism at the level of Chow motives by finite-dimensionality.
Then $\mathfrak{h}(X)$ and $\mathfrak{h}(Y)$ are isomorphic in 
$\sM_{rat}(\C)$, 
hence $M(X)$ and $M(Y)$ are isomorphic in $DM_{gm}(\C)$.
\end{proof} 

\begin{rk}\label{RemMukai}
Besides the properties of finite-dimensional objects, the other key point in the previous argument 
is the algebraicity of $\phi$.
This question goes back to a {\bf S{\u{a}}farev{\u\i}c's conjecture} stated at the ICM 1970 in Nice [Sh, B4 p. 416].
Shioda and Inose verified the conjecture in [S-I] for  singular K3 surfaces 
(those having the maximum possible Picard number, i.e. $\rho(X)=20$).
Then Mukai proved it in [Mu1, 1.10] for K3 surfaces with $\rho(X)\geq 11$,
and Nikulin showed its validity in [Ni, proof of Th.3] 
whenever $\textrm{NS}(X)$ contains a (nonzero) square zero element; 
this is is certainly the case if $\rho\geq 5$ and, 
according to Pjatetski{\u\i}-S{\u{a}}piro and S{\u{a}}farev{\u\i}c [PS-S],
it is equivalent to the existence of an elliptic fibration on $X$.
Eventually Mukai claimed to have completely solved the problem at ICM 2002 in Beijing [Mu2, Th. 2],
hence the hypothesis on the elliptic fibration could be removed.
\end{rk}

\section{Nikulin involutions}

Let  $X$ be  a smooth projective K3 surface  over $\C$ and let
$\Phi_{\mathcal A}\colon D^b(X) \iso D^b(X)$ be an autoequivalence. 
To  $\Phi_{\mathcal A}$  we can associate an Hodge isometry   
$$
\Phi^H_{\mathcal A} \colon \tilde H(X,\Z) \simeq \tilde H(X,\Z),
$$
as well as an automorphism of the Chow group 
$$
\Phi^{CH}_{\mathcal A} \colon CH^*(X) \simeq CH^*(X)
$$ 
induced by the correspondence $v^{CH}(\mathcal A) \in CH^*(X \times X)$ defined in [Hu2, 4.1]. 
We therefore get the two representations
\begin{diagram}
&&&&\mathrm{Aut} (CH^*(X))\\
&&\ruTo^{\rho^{CH}}&&\\
\mathrm{Aut} (D^b(X))& &&&\\
&&\rdTo_{\rho^H }&&\\
&&&&\mathcal O(\tilde H(X,\Z))\\
\end{diagram}

\noindent Here $\mathcal O(\tilde H(X,\Z))$ is the group of all integral Hodge isometries of the weight two Hodge structure defined on the Mukai lattice $\tilde H(X, \Z)$ and $\mathrm{Aut} (CH^*(X))$ denotes the group of  all  
automorphisms of the additive group $CH^*(X)$.
The following Theorem has been proved by D. Huybrects in [Hu1, 2.7].

\begin {thm}\label{TeoHuy}
$\Ker(\rho^H) = \Ker(\rho^{CH})$. 
\end{thm}   

\noindent 
From  Theorem \ref{TeoHuy}, if $\rho^H(\Phi_{\mathcal A}) = \Phi^H_{\mathcal A}$ is the identity in  
$\mathcal O(\tilde H(X,\Z))$,  then the correspondence $v^{CH}(\mathcal A)$ acts as the identity on $CH^*(X)$. 
In particular $\phi^H_{\mathcal A}$ acts as the identity on 
$H^{2,0}(X)\simeq H^0(X,\Omega^2_X) \subset H^2_{tr}(X,\C)$.  
The above Theorem suggested Huybrechts' conjecture \ref{ConjHuy}, that is
that any {\it symplectic} automorphism $f \in \textrm{Aut}(X)$ acting trivially on $H^{2,0}(X)$
acts trivially also on $CH^2(X)$.

In this section we deal with the case of a {\it symplectic involution}.
\medskip

\begin{defn}\label{DefNikInv} 
A {\bf Nikulin involution}  $\iota$ on a K3 surface $X $ is a symplectic involution, 
i.e.  $\iota^*(\omega)=\omega$ for all $\omega \in H^{2,0}(X)$. 
\end{defn} 

A Nikulin involution $\iota$ on a  complex projective K3 $X$ has the following special
properties, as proved by Nikulin (see e.g. [Mo, 5.2]):
\begin{itemize}
\item the fixed locus of $\iota$ consists of precisely eight distinct points and 
\item the minimal resolution $Y$ of the quotient $X/\iota=X/<\iota>$ is a K3 surface.
\end{itemize}
The surface $Y$ can also be obtained as the quotient of the blow up $\tilde X$ 
of $X$ in the $8$ fixed points by the extension $\tilde \iota$ of $\iota$ to $\tilde X$
([Mo, 3], [VG-S, 1.4]).
In other words we get the commutative diagram 
$$ 
\CD 
\tilde X @>{b}>>X \\
@V{g}VV   @VVV \\
 Y@>>> X/\iota
\endCD 
$$
where $Y$ is a desingularization of the quotient surface $X/\iota$ and $ Y \simeq \tilde X/\tilde \iota$, 
with $\tilde \iota$  the involution induced by $i$ on $\tilde X$. 
\par
\noindent 
As explained in [VG-S, 2.1] a K3 surface with a Nikulin involution has $\rho(X) \ge 9$.
Moreover ([VG-S, 2.4]) $\iota$ induces an isomorphism $\phi_\iota\colon T_{X,\Q} \iso T_{Y,\Q}$ 
of rational Hodge structures.
\par

Let $X$, $\tilde X$, and $Y$ be as in the diagram above and let $\mathfrak{t}_2(X)$ 
be the transcendental part of the motive of $X$. 
By [Ma, \S 3 Example 1] the degree $2$  map $g$ induces a splitting in $\sM_{rat}(\C)$  
$$
\mathfrak{h}(\tilde X) = (X,p) \oplus (X,\Delta_X -p) \simeq \mathfrak{h}(Y) \oplus (X, \Delta_X -p)
$$
where $p =1/2(\Gamma^t_g \circ \Gamma_g )\in A_{2}(X\times X)$.  
Since $\mathfrak{t}_2(-)$ is a birational  invariant we have $\mathfrak{t}_2(X) =\mathfrak{t}_2( \tilde X)$. 
From the above splitting it follows that $\mathfrak{t}_2(Y)$ is a direct summand of 
$\mathfrak{t}_2(X)$, i.e. $\mathfrak{t}_2(X) =\mathfrak{t}_2(Y) \oplus N$.

\begin{prop}\label{PropInvariance}
Let $X$, $\tilde X$, and $Y$ be as in the diagram above. 
Then 
$$ 
\mathfrak{t}_2(X) \simeq \mathfrak{t}_2(Y) \Longleftrightarrow A_0(X)^\iota =A_0(X)
$$
i.e. if and only if the involution $\iota$ acts as the identity on $A_0(X)$. 
If $\mathfrak{t}_2(X) \simeq \mathfrak{t}_2(Y)$, then the rational map $X \to Y$ induces an isomorphism 
between the motives $\mathfrak{h}(X)$ and $\mathfrak{h}(Y)$ and therefore also between 
$M(X)$ and $M(Y)$ in $DM_{gm}(\C)$.
\end{prop}

\begin{proof} 
Let $k(X)$ be the field of rational functions of $X$; 
then the Chow group of $0$-cycles on $X_{k(X)}$ may be identified with
$$ 
\textrm{lim} _{U \subset X}A^2(U \times X) \simeq A_0(X_{k(X)})
$$
where $U$ runs among the open sets of $X$ (see [Bl, Lecture 1. Appendix]).  
Since $\textrm{Alb}(X) =0$, the Albanese kernel $T(X_{k(X)})$ coincides  with $A_0(X_{k(X)})_0$. 
By [K-M-P, 5.10] there is an isomorphism  
$$
\textrm{End}_{\sM_{rat}}(\mathfrak{t}_2(X)) \simeq \frac {A_0(X_{k(X)})}{A_0(X)}
$$
where the identity map of $\mathfrak{t}_2(X)$ corresponds to the class of  $[\xi]$ in $\frac {A_0(X_{k(X)})}  {A_0(X)}$. 
Here $\xi$ denotes the generic point of $X$ and $[\xi]$ its class as a cycle in $A_0(X_{k(X)})$. 
The involution $\iota$ induces an involution $\bar \iota$ on $A_0(X_{k(X)})$.  
The splitting 
$$
[\xi] =1/2\left([\xi] +\bar \iota([\xi])\right) +1/2\left([\xi] -\bar \iota([\xi])\right)
$$
in  $A_0(X_{k(X)})$ corresponds to the splitting of the identity map of $\mathfrak{t}_2(X)$ in $\mathfrak{t}_2(X) =\mathfrak{t}_2(Y) \oplus N$. 
Therefore $N=0$ if and only if $\bar \iota([\xi])=[\xi]$. 
From the equalities $A_0(\mathfrak{t}_2(X))= A_0(X)_0$,  $A_0(\mathfrak{t}_2(Y))=A_0(Y)_0$ and $A_0(X)^\iota =A_0(Y)$ we get  
$$
\mathfrak{t}_2(X) \simeq \mathfrak{t}_2(Y)
\;\;\Longleftrightarrow\;\; 
N=0
\;\;\Longleftrightarrow\;\;
\bar \iota([\xi]) =[\xi] 
\;\;\Longleftrightarrow\;\;
A_0(X)^\iota =A_0(X).
$$
The rest follows from \ref{SubSecRCK}  because $X$ and $Y$ are K3 surfaces, with $\rho(X) =\rho(Y)$.
\end{proof}

Next we show that for every K3 surface with a Nikulin involution $\iota$ the finite dimensionality of 
$h(X)$ implies that $\iota$ acts as the identity on $A_0(X)$. 
Therefore for such $X$  Conjecture \ref{ConjHuy} holds true.
\par

\begin{lm}\label{Lemma}
Let $X$ be a K3 surface with a Nikulin involution $\iota$. 
Then $\rho (X) =\rho (Y)$ and $t=6$, 
where $t$ denotes the trace of the involution $\iota$ on $H^2(X,\C)$.
\end{lm}

\begin {proof} 
Let $X$ be a smooth projective surface over $\C$ with $q(X) =0$ and  an involution $\sigma$ 
and let $Y$ be a desingularization of $X/\sigma$. 
Let $e(-)$ be the topological Euler characteristic. 
Then we have (see [D-ML-P, 4.2])
$$
e(X) +t  +2 = 2e(Y) -2k
$$
where $t$ is the trace of the involution $\sigma$ on $H^2(X, \C)$ 
and $k$ is the number of  the isolated fixed points of $\sigma$. 
If $X$ and $Y$ are K3 surfaces  and $\sigma =\iota$ is  a Nikulin involution,  
then $e(X)=e(Y)=24$ and $k=8$. 
Therefore we get $t =6$. 
Since $\dim \ H^{2}_{\textrm{tr}}(X) = \dim \ H^{2}_{\textrm{tr}}(Y)$ and $b_2(X)=b_2(Y) =22$,  
we have $\rho(X) =\rho(Y)$.
\end{proof}

\begin{thm}\label{TeoNik}
Let $X$ be K3 surface with a Nikulin involution $\iota$. 
If $\mathfrak{h}(X)$ is finite dimensional then $\mathfrak{h}(X) \simeq \mathfrak{h}(Y)$,
therefore $\iota$ acts as the identity on $A_0(X)$.
\end{thm}

\begin{proof}   
Let $Y$ be the desingularization of $X/\iota$. 
Then $Y$ is a K3 surface and we have  $\mathfrak{t}_2(\tilde X) \simeq \mathfrak{t}_2(X)$ 
because $\mathfrak{t}_2(-)$ is a birational invariant for surfaces, see [K-M-P].    
Also 
$$
H^2_{\textrm{tr}}(X) \simeq H^2_{\textrm{tr}}( \tilde X) \simeq  H^2_{\textrm{tr}}(Y)
$$ 
because the Nikulin involution acts trivially on $H^2_{\textrm{tr}}(X)$.  
Let $E_i, 1 \le i \le 8$ be  the exceptional divisors of the blow-up $\tilde X \to X$ 
and let  $g_*(E_i) =C_i$ be the corresponding $(-2)$-curves on $Y$. 
We have $\rho = \textrm{rank}(\textrm{NS}(X)) \ge 9$,  $b_2(X) =b_2(Y)=22$ 
and  $e(X)=e(Y)=24$,  where $e(X)$ is the topological Euler characteristic.  
Let $t$ be the trace of the action of the involution $\iota$ on the vector space $H^2(X ,\C)$.
By Lemma \ref{Lemma} we have $t =6$. 
The involution $\iota$ acts trivially on $H^2_{tr}(X)$ 
which is a subvector space of $H^2(X,\C)$ of dimension $22-\rho$;  
therefore the trace of the action of $\iota$ on $\textrm{NS}(X)\otimes \C$ equals $\rho-16$. 
Since the only eigenvalues of  an involution are $+1$ and $-1$ we can find an orthogonal basis for 
$\textrm{NS}(X) \otimes \C$ of the form $H_1,\cdots, H_r;\, D_1, \cdots, D_8 $, 
with $r =\rho-8 \ge 1$  such that $\iota_*(H_j) =H_j$ and $\iota_*(D_l) =- D_l$. 
Then $\textrm{NS}(\tilde X)\otimes \C$ has a basis of the form 
$E_1,\cdots, E_8;\, H_1,\cdots, H_r;\, D_1, \cdots, D_8$.  
Since $X$ and $Y$ are K3 surfaces we have $q(X) =q(Y) =q(\tilde X)=0$. 
Therefore we can find Chow-K\"unneth decompositions for the motives  
$\mathfrak{h}(X)$, $\mathfrak{h}(\tilde X)$ such that $\mathfrak{h}_1=\mathfrak{h}_3=0$ and 
$$
\mathfrak{h}(X) 
\;\; =\;\;   
\un \oplus\;  {\mathfrak{h}_{2}}^\textrm{alg}(X)  \oplus \mathfrak{t}_2(X) \;\oplus \mathbb{L}^2 
\;\; \simeq\;\;  
\un \oplus\; \mathbb{L}^{\oplus \rho}\oplus \mathfrak{t}_2(X) \;\oplus \mathbb{L}^2 
$$
$$ 
\mathfrak{h}(\tilde X)
\;\; =\;\; 
\un \oplus\; 
{\mathfrak{h}_{2}}^\textrm{alg}(\tilde X) \oplus \mathfrak{t}_2(X)
 \;\oplus \mathbb{L}^2 
\;\;\simeq\;\;  
\mathfrak{h}(X) \oplus \mathbb{L}^{\oplus 8}
$$ 
where ${\mathfrak{h}_2}^\textrm{alg}(\tilde X)=  (\tilde X,\pi^{\textrm{alg}}_2(\tilde X)) $  with $\pi^{\textrm{alg}}_2(\tilde X)=\Gamma   + I $ and  
$$
\Gamma =\sum_{1 \le k \le 8} \frac  { [E_k \times E_k]}{E^2_k} + \sum _{1 \le j \le r }\frac {[H_j \times H_j]}{H^2_j},
\quad  I =\sum_{1 \le h \le r} \frac {[D_h \times D_h]}{D^2_h}. 
$$
Also  
$$  
\mathbb{L}^{\oplus 8} \simeq \left(\tilde X,  \sum_{1 \le k \le 8} \frac  { [E_k \times E_k]}{E^2_k} \right)
$$
Let $g \colon\tilde X \to Y$ and let  $p =1/2(\Gamma^t_g\circ \Gamma_g )\in A^2(\tilde X \times \tilde X)$: 
then $p$ is a projector and 
$$
\mathfrak{h}(\tilde X) =(\tilde X, p)\oplus (\tilde X , \Delta_{\tilde X} -p) \simeq \mathfrak{h}(Y) \oplus (\tilde X , \Delta_{\tilde X}-p)  
$$
because $(\tilde X, p)\simeq \mathfrak{h}(Y) $ by [Ma, \S 3 Example 1].  
The set of  $r +8=\rho$ divisors $g_*(E_k)=C_k $,  for $1 \le k \le  8$ and $g_*(H_ j) \simeq H_j $, for $ 1 \le  j \le r$  
gives an orthogonal basis for $\textrm{NS}(Y)\otimes \Q$. 
Therefore we can find a Chow-K\"unneth decomposition of $\mathfrak{h}(Y)$ such that  
$$
{\mathfrak{h}_2}^\textrm{alg}(Y) \simeq   (\tilde X , \Gamma)\simeq \mathbb{L}^{\oplus \rho}
$$
and we get
$$
\mathfrak{h}_2(\tilde X) 
\;\;=\;\; 
{\mathfrak{h}_2}^\textrm{alg}(\tilde X)\oplus \mathfrak{t}_2(X)  
\;\;\simeq\;\;  
{\mathfrak{h}_2}^\textrm{alg}(Y) \oplus \mathbb{L}^{\oplus 8} \oplus \mathfrak{t}_2(Y)\oplus M  
$$
where $H^*(M)=0$ because $H^2_{\textrm{tr}}(\tilde X) =H^2_{\textrm{tr}}(X)=H^2_{\textrm{tr}}(Y)$. 
From Theorem \ref{TeoConservativity} it follows  that $M=0$ and we get an isomorphism
$$
\mathfrak{h}_2(\tilde X)  
\;\;\simeq\;\;   
\mathfrak{h}_2(Y)\oplus \mathbb{L}^{\oplus 8} 
\;\;\simeq\;\;  
\mathfrak{h}_2(X) \oplus  \mathbb{L}^{\oplus 8}
$$
\noindent  
which implies $\mathfrak{h}(X) \simeq \mathfrak{h}(Y)$. 
The rest follows from Proposition \ref{PropInvariance}.
\end{proof}

The following result  gives examples of K3 surfaces with a  Nikulin involution $\iota$  
such that  $\iota$ acts as the identity on $A_0(X)$.

\begin {thm}\label{TeoRhoHigh} 
Let $X$ be a smooth projective K3 surface over $\C$ with $\rho(X) =19,20$. 
Then $X$ has a Nikulin involution $\iota$,  
$\mathfrak{h}(X)$ is finite dimensional and $\iota$ acts as the identity on $A_0(X)$.
\end{thm}

\begin {proof}  
By [Mo, 6.4]  $X$ admits a Shioda-Inose structure, i.e. there is a Nikulin involution $\iota$ on $X$ such that 
the desingularization $Y$ of the quotient surface $X/\iota$ is a Kummer surface, associated to an abelian surface $A$;
hence $\mathfrak{h}(Y)$ is finite dimensional by [Pe1,  5.8].
The rational map $f : X \to Y$ induces a splitting $\mathfrak{t}_2(X) \simeq \mathfrak{t}_2(Y) \oplus N$. 
Since $\mathfrak{t}_2(Y) $ is finite dimensional  we are left to show that $N=0$. 
By the same argument as in the proof of Proposition \ref{PropInvariance} the vanishing of $N$ is equivalent to $A_0(X)^\iota=A_0(Y)$.  
By  [Mo, 6.3 (iv)]  the Neron Severi group of $X$ contains the sublattice $E_8(-1)^2$.  
Hence by the results in [Hu2,  6.3, 6.4] the  symplectic automorphism $\iota$ acts as the identity on $A_0(X)$.
As, by [K-M-P, 6.13], we have  $\mathfrak{t}_2(Y) =\mathfrak{t}_2(A)$, the motive $\mathfrak{h}(X) $ is finite dimensional and it lies  
in the subcategory  of $\sM_{\textrm{rat}}(\C)$ generated by the motives of abelian varieties.  
\end{proof}

The next theorem gives examples of surfaces $X$ and $Y$ such that $M(X) \simeq M(Y)$ 
but the derived categories $D^b(X)$ and $D^b(Y)$ are not equivalent. 
We will use the following result by Van Geemen and Sarti 

\begin{prop}\label{PropVG-S}([VG-S 2.5]) 
Let $X$ be a complex K3 surface with a Nikulin involution $\iota$ 
and  let  $Y$ be a desingularization of the quotient surface $X/\iota$. 
The involution induces an isomorphism of Hodge structures between  $T_{X,\Q}$ and $T_{Y,\Q}$. 
If the dimension of the  $\Q$-vector space $T_{X,\Q}$ is odd 
there is no isometry between $T_{X,\Q}$ and $T_{Y,\Q}$.
\end{prop}

\begin{thm}\label{TeoRho9}
Let $X$ be a complex K3 surface with a Nikulin involution $\iota$ such that $\rho(X)=9$ 
and let $Y$ be the desingularization of $X/\iota$.
Assume that the map $f \colon X \to Y$ induces an isomorphism between $\mathfrak{t}_2(X) $ and $\mathfrak{t}_2(Y)$. 
Then $\iota$ acts as the identity on $A_0(X)$, 
the rational map $f \colon X \to Y$  induces an isomorphism  $M(X)\iso M(Y)$ in $DM_{\textrm{gm}}(\C)$,  
but  the isomorphism of Hodge structures $\phi_\iota \colon T_{X,\Q} \to T_{Y,\Q}$ is not an isometry.
\end{thm}

\begin {proof}
The Nikulin involution $\iota$ induces an isomorphism of Hodge structures 
$\phi_\iota \colon T_{X,\Q} \to T_{Y,\Q}$ 
which by Proposition \ref{PropVG-S} is not an isometry because $\dim\, T_{X,\Q} = 22 -9 $ is odd. 
Since $X$ and $Y$ are both K3 surfaces the isomorphism $\mathfrak{t}_2(X) \simeq \mathfrak{t}_2(Y)$ implies 
$\mathfrak{h}(X) \simeq \mathfrak{h}(Y)$ in $\sM^{\textrm{eff}}_{\textrm{rat}}(\C)$, hence also $M(X) \simeq M(Y)$.
\end{proof}
 
\begin{exs}\label{ExamplesVG-S}  
The following are  examples of K3 surfaces $X$ with a Nikulin involution $\iota$ and $\rho(X)=9$ 
such that $\mathfrak{t}_2(X)\simeq \mathfrak{t}_2(Y)$ hence $\mathfrak{h}(X) \simeq \mathfrak{h}(Y)$.
Therefore $X$ satisfies Huybrechts' conjecture \ref{ConjHuy}, i.e. $\iota$ acts as the identity on $A_0(X)$.
On the other hand, $X$ and $Y$ are not Fourier-Mukai partner because, as in Theorem \ref{TeoRho9},
there is no Hodge isometry between their transcendental lattices.
The proof  of the isomorphism $\mathfrak{t}_2(X) \simeq \mathfrak{t}_2(Y)$ in  these cases 
follows directly from the geometric description of $X$ and $Y$ given by Van Geemen and Sarti in [VG-S], see [Pe2].
 
(i) $X$ a double cover of $\P^2$ branched over a sextic curve and $Y$ a double cover of a quadric cone in $\P^3$; 
\par

(ii) $X$ is a double cover of a quadric in $\P^3$ and $ Y$ is the double cover of $\P^2$ branched over a reducible sextic;\par

(iii) $X$ is the intersection of 3 quadrics in $\P^5$ and $Y$ is a quartic surface in $\P^3$.
\end{exs}


\begin{thebibliography}{} 

\bibitem[An1]{An1}
Y. Andr\'e  
{\it Une introduction aux motifs}
Panoramas et Synth\'eses, 17. 
Soci\'et\'e Math\'ematique de France, Paris, 2004.

\bibitem[An2]{An2}
Y. Andr\'e
{\it Motifs de dimension finie (d'apr\`es S.-I. Kimura, P. O'Sullivan$\dots$)} 
S\'eminaire Bourbaki. Vol. 2003/2004. 
Ast\'erisque No. 299 (2005), Exp. No. 929, viii, 115-145.

\bibitem[B-B-HR]{B-B-HR}
C. Bartocci, U. Bruzzo, D. Hern\'andez Ruip\'erez 
{\it Fourier-Mukai and Nahm transforms in geometry and mathematical physics} 
Progress in Mathematics, 276. Birkh\"auser Boston, Inc., Boston, MA, 2009.

\bibitem[B-L]{B-L}
C. Birkenhake, H. Lange 
{\it Complex abelian varieties} 
Second edition. Grundlehren der Mathematischen Wissenschaften, 302. 
Springer-Verlag, Berlin, 2004. 

\bibitem[Bl]{Bl}
S. Bloch 
{\it Lectures on algebraic cycles}
Duke University Mathematics Series IV, Duke University Press, Durham U.S.A., (1980).
 
\bibitem[Br]{Br} 
T. Bridgeland 
{\it Fourier-Mukai transforms for elliptic surfaces} 
J. Reine Angew. Math. 498 (1998), 115-133.
 
\bibitem[B-M]{B-M} 
T. Bridgeland, A. Maciocia 
{\it Complex surfaces with equivalent derived categories} 
Math. Z. 236 (2001), no. 4, 677-697. 

\bibitem[De]{De}
P. Deligne {\it La conjecture de Weil pour les surfaces K3}
Invent. Math. 15 (1972), 206-226.
 
\bibitem[DP-M1]{DP-M1} 
A. Del Padrone, C. Mazza 
{\it Schur finiteness and nilpotency} 
C. R. Math. Acad. Sci. Paris 341 (2005), no. 5, 283-286. 

\bibitem[DP-M2]{DP-M2} 
A. Del Padrone, C. Mazza
{\it Schur-finiteness and endomorphisms universally of trace zero via certain trace relations} 
Comm. Algebra 37 (2009), no. 1, 32-39.

\bibitem[DP]{DP} 
A. Del Padrone 
{\it A  note on derived equivalence and finite dimensional motives} 
in preparation. 
 
\bibitem[D-M]{D-M} 
C. Deninger, J. Murre 
{\it Motivic decomposition of abelian schemes and the Fourier transform} 
J. Reine Angew. Math. 422 (1991), 201-219.
 
\bibitem[D-ML-P]{D-ML-P} 
I. Dolgachev, M. Mendes Lopez and R. Pardini 
{\it Rational surfaces with many nodes}
Compositio Math. 132 (2002), no. 3, 349-363.

\bibitem[VG-S]{VG-S}
B. van Geemen and A. Sarti 
{\it Nikulin Involutions on K3 Surfaces}
Math. Z. (2007), 731-753.
   
\bibitem[G-G]{G-G} 
S. Gorchinskiy and V. Guletski{\u\i} 
{\it Motives and representability of algebraic cycles on threefolds over a field}
arXiv:0806.0173v2 [math.AG].

\bibitem[G-P]{G-P} 
V. Guletski{\u\i} and C. Pedrini 
{\it Finite-dimensional Motives and the Conjectures of Beilinson and Murre}
$K$-Theory {\bf 30} (2003),  243-263.  

\bibitem[Hu]{Hu}
D. Huybrechts,
{\it  Fourier-Mukai transforms in algebraic geometry}
Oxford Mathematical Monographs. The Clarendon Press, Oxford University Press, Oxford, 2006. 

\bibitem[Hu1]{Hu1}  
D. Huybrechts 
{\it Chow groups of K3 surfaces and spherical objects} 
J. Eur. Math. Soc. 12 (2010), no. 6, 1533-1551.
 
\bibitem[Hu2]{Hu2} 
D. Huybrechts 
{\it Chow groups and derived categories of K3 surfaces} 
to appear in ``Proc. Classical Algebraic Geometry today'', MSRI January 2009, 
arXiv:0912.5299v1 [Math.AG].

\bibitem[H-NW]{H-NW}
D. Huybrechts, M. Nieper-Wisskirchen
{\it Remarks on derived equivalences of Ricci-flat manifolds}
Math. Z. (2011) 267:939-963

\bibitem[In]{In}
H. Inose 
{\it Defining equations of singular $K3$ surfaces and a notion of isogeny} 
in ``Proceedings of the International Symposium on Algebraic Geometry'' (Kyoto Univ., Kyoto, 1977), 
pp. 495-502, Kinokuniya Book Store, Tokyo, 1978.
 
 \bibitem[Ja]{Ja}
U. Jannsen 
{\it On finite-dimensional motives and Murre's conjecture} 
Algebraic cycles and motives. Vol. 2, 112-142, London Math. Soc. Lecture Note Ser., 
344, Cambridge Univ. Press, Cambridge, 2007.
 
\bibitem[Ka]{Ka}
Y. Kawamata 
{\it $D$-equivalence and $K$-equivalence}
J. Differential Geom. 61 (2002), no. 1, 147-171.
 
\bibitem[Ki]{Ki} 
 S. I. Kimura 
 {\it Chow groups can be finite-dimensional, in some sense}
 Math. Ann. {\bf 331} (2005), 173-201.

\bibitem[K-M-P]{K-M-P} 
B. Kahn, J. Murre and C. Pedrini 
{\it On the transcendental part of the motive of a surface}
in ``Algebraic cycles and Motives'' Vol. II, London Math. Soc. LNS 
{\bf 344} (2008), Cambridge University Press, 1-58. 

\bibitem[L-S-Y]{L-S-Y}
R. Livn\'e, M. Sch\"utt, N. Yui 
{\it The modularity of K3 surfaces with non-symplectic group actions} 
Math. Ann. 348 (2010), no. 2, 333-355.

\bibitem[Ma]{Ma} 
Yu. I. Manin 
{\it Correspondences, motives and monoidal transformations}
Math. USSR Sb. {\bf 6} (1968), 439-470.

\bibitem[Maz]{Maz}
C. Mazza 
{\it Schur functors and motives} 
$K$-Theory 33 (2004), no. 2, 89-106. 

\bibitem[Mo]{Mo} 
D. R. Morrison 
{\it On K3 surfaces with large Picard number} 
Inv. Math. {\bf 75} (1984), 105-121.

\bibitem[Mu1]{Mu1}
S. Mukai 
{\it On the moduli space of bundles on a K3 surface I } in ``Vector bundles on algebraic varieties'',
Tata Inst. of Fund. Research Stud. Math. 11 (1987), 34-413.

\bibitem[Mu2]{Mu2} 
S. Mukai 
{\it Vector bundles on a K3 surface} 
in ``Proceedings of the International Congress of Mathematicians Vol. II (Beijing, 2002)'', 495-502, 
Higher Ed. Press, Beijing, 2002. 
 
\bibitem[Ni]{Ni} 
V. Nikulin 
{\it  On correspondences between surface of K3 type} 
Math. USSR Izvestyia {\bf 30}  (1988), no. 2, 375-383.

\bibitem[Or1]{Or1} 
D. Orlov 
{\it Equivalence of derived categories and K3 surfaces} 
in ``Algebraic geometry, 7'', J. Math. Sci. (New York) 84  (1997), no. 5, 1361-1381.

\bibitem[Or2]{Or2} 
D. Orlov 
{\it Derived categories of coherent sheaves and motives} 
Usp. Mat. Nauk 60 (6), 23-232 (2005) [Russ. Math. Surv. 60, 1242-1244 (2005)]; 
arXiv: math/0512620.

\bibitem[Pe1]{Pe1} 
C. Pedrini 
{\it On the motive of a K3 surface} 
in ``The geometry of algebraic cycles'', Clay Math. Proc., {\bf 9 }, (2010), 53-74.

\bibitem[Pe2]{Pe2}  
C. Pedrini 
{\it The Chow Motive of a K3 surface} 
Milan J. Math. 77 (2009), 151-170.

\bibitem[Pe3]{Pe3}
C. Pedrini 
{\it On the finite-dimensionality of a K3 surface} 
submitted.

\bibitem[PS-S]{PS-S}
I. I. Pjatetski{\u\i}-S{\u{a}}piro, I. R. S{\u{a}}farev{\u\i}c 
{\it A Torelli theorem for algebraic surfaces of type K3} 
Izv. AN SSSR. Ser. mat., 35 (1971), no. 3, 530-572; 
English transl.: Math. USSR Izv. 5 (1971), no. 3, 547-588.

\bibitem[Ro]{Ro}
R. Rouquier 
{\it Cat\'egories d\'eriv\'ees et g\'eom\'etrie birationnelle (d'apr\'es Bondal, Orlov, Bridgeland, Kawamata et al.)}
S\'eminaire Bourbaki. Vol. 2004/2005. Ast\'erisque No. 307 (2006), Exp. No. 946, viii, 283-307.

\bibitem[Sh]{Sh}
I. R. S{\u{a}}farev{\u\i}c 
{\it Le th\'eor\`eme de Torelli pour les surfaces alg\'ebriques de type $K3$} 
in ``Actes du Congr\`es International des Math\'ematiciens'' (Nice, 1970), Tome 1, pp. 413-417. 
Gauthier-Villars, Paris, 1971. 

\bibitem[S-I]{S-I}
T. Shioda, H. Inose 
{\it On singular $K3$ surfaces} in ``Complex analysis and algebraic geometry'', pp. 119-136. 
Iwanami Shoten, Tokyo, 1977.

\bibitem[S-K]{S-K}
T. Shioda, T. Katsura {\it On Fermat varieties} 
T\^ohoku Math. J. (2) 31 (1979), no. 1, 97-115.

\bibitem[So]{So}
P. Sosna, {\it Derived equivalent conjugate K3 surfaces} 
Bull. Lond. Math. Soc. 42 (2010), no. 6, 1065-1072.

\bibitem[Vo]{Vo}
V. Voevodsky 
{\it Triangulated categories of motives over a field } 
in ``Cycles, transfers and motivic homology theories'', Ann. of Math. Stud. {\bf 143}  (2000), 188-238.
 
\end{thebibliography}
 \end{document}